\documentclass[oneside,a4paper,11pt,notitlepage]{article}
\usepackage[left=0.8in, right=0.8in, top=1.5in, bottom=1.5in]{geometry}

\usepackage[english]{babel}
\usepackage{amssymb}
\usepackage{lipsum} 
\usepackage{lmodern}
\usepackage{amsmath}
\usepackage{amsthm}
\usepackage{bm}
\usepackage{mathtools}
\usepackage{braket}
\usepackage{esint}

\usepackage{relsize}
\usepackage{minted}
\usepackage{mathtools}

\mathtoolsset{showonlyrefs}
\usepackage{enumitem}

\usepackage{enumitem}
\usepackage{booktabs}
\usepackage{graphicx}
\usepackage{tikz}
\usetikzlibrary{patterns}
\usepackage{multicol}
\usepackage{caption}
\usepackage{bigints}
\usepackage[skins,theorems]{tcolorbox}

\usepackage{listings}
\usepackage{xcolor}

\definecolor{codegreen}{rgb}{0,0.6,0}
\definecolor{codegray}{rgb}{0.5,0.5,0.5}
\definecolor{codepurple}{rgb}{0.58,0,0.82}
\definecolor{backcolour}{rgb}{0.95,0.95,0.92}

\lstdefinestyle{pythonstyle}{
    backgroundcolor=\color{backcolour},   
    commentstyle=\color{codegreen},
    keywordstyle=\color{magenta},
    numberstyle=\tiny\color{codegray},
    stringstyle=\color{codepurple},
    basicstyle=\ttfamily\footnotesize,
    breakatwhitespace=false,         
    breaklines=true,                 
    captionpos=b,                    
    keepspaces=true,                 
    numbers=left,                    
    numbersep=5pt,                  
    showspaces=false,                
    showstringspaces=false,
    showtabs=false,                  
    tabsize=2,
    language=Python
}

\tcbset{highlight math style={enhanced,
colframe=black,colback=white,arc=0pt,boxrule=1pt}}
\def\XXint#1#2#3{{\setbox0=\hbox{$#1{#2#3}{\int}$}
    \vcenter{\hbox{$#2#3$}}\kern-.5\wd0}}

\newcounter{theor}

\theoremstyle{definition}

\theoremstyle{plain}
\newtheorem{teorema}{Theorem}[section]

\newtheorem{lemma}[teorema]{Lemma}
\newtheorem{prop}[teorema]{Proposition}
\newtheorem{corollario}[teorema]{Corollary}
\newtheorem{conjecture}[teorema]{Conjecture}
\theoremstyle{definition}
\newtheorem{esempio}{Example}[section]
\newtheorem{oss}[esempio]{Remark}

\newtheorem*{open*}{Open problems}

\makeatletter
\newcommand{\myfootnote}[2]{\begingroup
	\def\@makefnmark{}%
	\addtocounter{footnote}{-1}%
	\footnote{\textbf{#1} #2}
	\endgroup}
\makeatother

\def\esc#1{\left\langle #1\right\rangle}
\definecolor{mygreen}{RGB}{0, 255, 0}

\usepackage{hyperref}
\hypersetup{linktoc=none, bookmarksnumbered, colorlinks=true, linkcolor=magenta, citecolor=cyan}

\title{Remarks on Brunn-Minkowski-type inequalities related to the Ornstein-Uhlenbeck operator}
\author{Francisco Marín Sola and Francesco Salerno}
\date{}

\newcommand{\Addresses}{{
\bigskip 
  
   \medskip

      F. Marín Sola (corresponding author): \texttt{francisco.marin7@um.es} 
   \medskip 

 \textsc{Department of Engineering and Technology of Computers, area of Applied Mathematics. Universidad de Murcia, Spain.}
 
 \medskip

    F. Salerno: \texttt{f.salerno@ssmeridionale.it} 
  
   \medskip 
 \textsc{Mathematical and physical sciences for advanced materials and technologies, Scuola superiore meridionale, Largo San Marcelino 10, 80138 Napoli, Italy.}

 \par\nopagebreak 

}} 

\begin{document}
\maketitle
\begin{abstract}
We investigate Brunn-Minkowski-type inequalities for the torsional rigidity $T_\gamma$ and the first eigenvalue $\lambda_\gamma$ associated with the Ornstein-Uhlenbeck operator. Counterexamples are provided showing that neither concavity nor convexity properties hold for $T_\gamma$ on general bounded convex sets.  We also demonstrate that log-concavity and log-convexity properties fail in this setting. In the case of centrally symmetric sets, we answer a question raised by Cordero-Erausquin and Eskenazis by showing that $T_\gamma^{1/(n+2)}$ is neither convex nor concave.  On the positive side, we prove that $T_\gamma^{1/3}$ is convex with respect to Minkowski addition when restricted to Euclidean balls centered at the origin. For $\lambda_\gamma$, we answer negatively a question posed by Colesanti, Francini, Livshyts, and Salani by showing that the inequality $\lambda_\gamma(\Omega_t)^{-1/2} \geq (1-t)\lambda_\gamma(\Omega_0)^{-1/2} + t\lambda_\gamma(\Omega_1)^{-1/2}$ does not hold, even for centrally symmetric sets. 
\end{abstract}

\section{Introduction}\label{Introduction}
    Let $K_0$, $K_1$ be convex bodies (compact and convex subsets of $\mathbb{R}^n$ with non-empty interior) and  $t\in[0,1]$. The Brunn-Minkowski inequality states that
    \begin{equation}\label{BMvol}
        |K_t|^\frac{1}{n}\geq (1-t)|K_0|^\frac{1}{n}+t|K_1|^\frac{1}{n},
    \end{equation}
    where $K_t=(1-t)K_0+tK_1=\{(1-t)x+ty\,:\, x\in K_0,\, y\in K_1\}$. Here $|\cdot|$ stands for the $n$-dimensional Lebesgue measure or \emph{volume}. Equality holds in \eqref{BMvol} if and only if $K_0$ and $K_1$ are homothetic. For an extensive survey article on this and related inequalities we refer the interested reader to \cite{G20}.
    
    Inequalities of the same type as \eqref{BMvol}, where the volume is replaced by other functionals, arise in many different contexts (see e.g. \cite[Section~7]{Sh1}). In particular, for those coming from Calculus of Variations, a systematic study was given by Colesanti in \cite{C}. More generally, given a functional
    \begin{equation*}
        F:\mathcal{K}^n\rightarrow(0,+\infty),
    \end{equation*}
    where $\mathcal{K}^n$ is the class of all compact, convex subsets of $\mathbb{R}^n$, which is homogeneous of order $\alpha\not=0$, we say that $F$ satisfies a Brunn-Minkowski inequality if
    \begin{equation}\label{GenBM}
        F((1-t)K_0+tK_1)^\frac{1}{\alpha}\geq(1-t)F(K_0)^\frac{1}{\alpha}+tF(K_1)^\frac{1}{\alpha},
    \end{equation}
    for all $K_0,\,K_1\in\mathcal{K}^n$ and for all $t\in[0,1]$. Due to the homogeneity, to prove \eqref{GenBM}, it suffices to show that $F$ is quasi-concave (quasi-convex) if $\alpha>0$ ($\alpha<0$), namely
    \begin{equation}\label{q.c.}
        \begin{split}
            F((1-t)K_0+tK_1)&\geq\min\{F(K_0),\,F(K_1)\} \quad \text{ if } \alpha>0,\\
            F((1-t)K_0+tK_1)&\leq\max\{F(K_0),\,F(K_1)\} \quad \text{ if } \alpha<0.
        \end{split}
    \end{equation}

    In this paper, we are concerned with some geometric quantities related to the Ornstein-Uhlenbeck operator, namely the torsional rigidity $T_\gamma$ and the first eigenvalue $\lambda_\gamma$ associated to this operator (see \S \ref{Sec2} for the precise definition). Such quantities, usually referred to as the Gaussian torsional rigidity and the Gaussian eigenvalue, represent the counterpart, in Gauss space, of the torsional rigidity and principal frequency for the Laplace operator, whose corresponding Brunn-Minkowski inequalities were proved in \cite{Borell2} and \cite{BL76} respectively (see also \cite{C}). 
    
    In a recent work \cite{CFLS}, a Brunn-Minkowski-type inequality for $\lambda_\gamma$ is established. Specifically, they prove that the first eigenvalue is convex with respect to the Minkowski addition (see also \cite{CQS} for a different proof of this and \cite{CQS2} for the extension to the nonlinear setting);
    \begin{teorema}[Theorem 1.2, \cite{CFLS}]\label{t:BMeig}Let $\Omega_0$, $\Omega_1$ be open, bounded subsets of $\mathbb{R}^n$. Let $t\in[0,1]$. Then
    \begin{equation}\label{BMeig}
        \lambda_\gamma(\Omega_t)\leq(1-t)\lambda_\gamma(\Omega_0)+t\lambda_\gamma(\Omega_1),
    \end{equation}
    where
    \begin{equation*}
        \Omega_t=(1-t)\Omega_0+t\Omega_1.
    \end{equation*}
    \end{teorema}
    In the case of the first eigenvalue of the Laplacian, an inequality of the type \eqref{BMeig} implies the concavity of the functional $\lambda^{-\frac{1}{2}}(\,\cdot\,)$ thanks to its homogeneity. In contrast, due to the lack of it, an exponent Brunn-Minkowski inequality does not follow from \eqref{BMeig}. A similar phenomenon happens for the Gaussian measure; although a Brunn-Minkowski-type inequality follows from Borell's characterization \cite{Borell}, namely
    \begin{equation*}
        \gamma(\Omega_t)\geq \gamma(\Omega_0)^{1-t}\gamma(\Omega_1)^t,
    \end{equation*}
    an exponent version is not immediately deduced. Indeed, it was conjectured in \cite{GarZv} that an exponent version of the previous inequality, i.e., the inequality for $\gamma^{1/n}$, could be true for convex sets containing the origin. A counterexample to the latter was provided by Nayar and Tkocz in \cite{NayTko}, where the possibility of such an inequality for centrally symmetric convex sets remained open. This result was then achieved by Eskenazis and Moschidis in \cite{EskMos} (see also \cite{KolLiv} and the references therein for crucial contributions towards proving the conjecture). So a natural question, also posed in \cite{CFLS}, is whether a similar improvement holds for \eqref{BMeig}. We answer this question in the negative, at least if it is understood as a generalization of what happens in the Euclidean case. In particular, we show with a counterexample that an inequality of the type
    \begin{equation*}
        \lambda_\gamma(\Omega_t)^{-\frac{1}{2}}\geq (1-t)\lambda_\gamma(\Omega_0)^{-\frac{1}{2}}+t\lambda_\gamma(\Omega_1)^{-\frac{1}{2}}
    \end{equation*}
    does not hold even for centrally symmetric sets. 
    
    In view of the Euclidean case and Theorem \ref{t:BMeig}, it is also natural to wonder  whether a Brunn-Minkowski-type inequality holds true for the Gaussian torsional rigidity. In more detail, the following questions arise:
    \begin{itemize}
        \item [\textbf{(Q1)}] Do concavity or convexity properties hold for the Gaussian torsional rigidity with respect to the Minkowski sum for bounded open convex sets?
        \item [\textbf{(Q2)}] Does the case of centrally symmetric sets allow us to say anything more?
    \end{itemize}
    
    In this paper, we answer these questions as follows. Referring to (Q1), we show with two examples that both concavity and convexity properties cannot be true in general for open, bounded convex sets. Concerning question (Q2), and addressing \cite[Question 18]{CE}, we provide a counterexample showing that the functional $T_\gamma^{1/(n+2)}$ is neither concave nor convex for centrally symmetric convex sets. Specifically, considering $\Omega_i=B_{R_i}$ (an Euclidean ball of radius $R_i$ centered at the origin), $i=0,1$, we find that, in dimension $n\geq3$, no concavity or convexity property holds for the Gaussian torsion with exponent $1/(n+2)$. This will be the content of Example \ref{Example-Tor-Exp}.

    In light of the latter, one may wonder about the smallest exponent $\alpha$ such that the functional $T_\gamma^{\alpha}$ is convex (or concave) for centrally symmetric convex sets. We provide a partial answer; we show that, for Euclidean balls, the functional $T_\gamma^{1/3}$ is  convex with respect to the Minkowski addition. This is the content of the following theorem.
    \begin{teorema}\label{Teo1}
        Consider $R_0,\,R_1>0$, then
        \begin{equation}\label{BM}
            T_\gamma(B_{R_t})^\frac{1}{3}\leq (1-t)T_\gamma(B_{R_0})^\frac{1}{3}+tT_\gamma(B_{R_1})^\frac{1}{3},
        \end{equation}
        where $R_t=(1-t)R_0+tR_1$. Equality holds if and only if $R_0 = R_1$. Moreover the exponent is sharp.
    \end{teorema}
    Regarding the sharpness of the exponent, one may ask about the smallest dimensional exponent $\alpha_n$ such that the functional $T_\gamma(B_R)^{\alpha_n}$ is convex. The answer to this question follows immediately from the proof of Theorem \ref{Teo1}. For the sake of completeness, we collect it in the next corollary. 
    \begin{corollario}
        Consider $R_0$, $R_1>0$, then
        \begin{equation*}
            T_\gamma(B_{R_t})^{\alpha_n}\leq (1-t)T_\gamma(B_{R_0})^{\alpha_n}+tT_\gamma(B_{R_1})^{\alpha_n},
        \end{equation*}
        where $R_t=(1-t)R_0+tR_1$ and
        \begin{equation*}
            \alpha_n=\sup_{R>0}\left\{\frac{R^{2n-2}e^{-R^2}}{\left(\int_0^R s^{n-1}e^{-\frac{s^2}{2}}\,ds \right)^3}\int_0^R\frac{\left(\int_0^t s^{n-1}e^{-\frac{s^2}{2}}\,ds \right)^2}{t^{n-1}e^{-\frac{t^2}{2}}}\,dt\left[\frac{n-1-R^2}{R^ne^{-\frac{R^2}{2}}}\int_0^R s^{n-1}e^{-\frac{s^2}{2}}\,ds-2 \right]+1\right\}.
        \end{equation*}
    \end{corollario}
    We observe in pass that, if $T_\gamma^\alpha$ is convex for some $\alpha >0$, then $T_\gamma^\beta$ is convex for any $\beta >\alpha$. The following relationship between the various exponent discussed above can be inferred from the proof of Theorem \ref{Teo1}:
        \begin{equation*}
            \frac{1}{n+2}\leq \alpha_n\leq \sup_{n\geq1}\alpha_n=\frac{1}{3}.
        \end{equation*}
        Note that when $n=1$ the first inequality holds as an equality. Numerical evidence suggest equality when $n = 2$, whereas the inequality is strict for $n\geq3$. 
    
    The proof of a more general result seems to require a greater understanding of why convexity fails in the general case and how central symmetry comes into play. In light of Theorem \ref{Teo1} and the counterexamples found, it may be reasonable to conjecture the following.
    \begin{conjecture}
        Let $\Omega_0$, $\Omega_1$ be two open, bounded and centrally symmetric subsets of $\mathbb{R}^n$. Then
        \begin{equation*}
            T_\gamma(\Omega_t)^\frac{1}{3}\leq (1-t)T_\gamma(\Omega_0)^\frac{1}{3}+tT_\gamma(\Omega_1)^\frac{1}{3},
        \end{equation*}
        where $\Omega_t = (1-t)\Omega_0 + t \Omega_1$. Equality holds if and only if $\Omega_0 = \Omega_1$.
    \end{conjecture}
    
    Finally, the case of unbounded sets is investigated. In the next result, we show that, for a particular family of unbounded sets, namely half-spaces, the Gaussian torsional rigidity is also convex.
    \begin{prop}\label{TeoHalfSp}
        Let $s\in\mathbb{R}$ and consider the half-space
        \begin{equation*}
            H_s=\{x=(x_1,\dots,x_n)\in\mathbb{R}^n\,:\, x_1\geq s\}.
        \end{equation*}
        Then, for all $s_0,\,s_1\in\mathbb{R}$
        \begin{equation*}
            T_\gamma(H_{s_t})\leq(1-t)T_\gamma(H_{s_0})+tT_\gamma(H_{s_1}),
        \end{equation*}
        where $s_t=(1-t)s_0+ts_1$.
    \end{prop}
    The paper is organized as follows. In Section \ref{Sec2} we recall some preliminary tools about the Gaussian measure and define the Gaussian eigenvalue and torsional rigidity. In Section \ref{Counterexample} we discuss the counterexamples: first for the Gaussian torsional rigidity, answering question \textbf{(Q1)}; then for its exponent version, and finally for  the exponent version of the Brunn-Minkowski inequality for the Gaussian eigenvalue. Section \ref{Proof of the results} is devoted to the proof of Theorem \ref{Teo1}, while that of Proposition \ref{TeoHalfSp} is deferred to Appendix \ref{Appendix}. We conclude in Appendix \ref{Appendix} by discussing the numerical evidence, when $n=2$, for the Brunn-Minkowski inequality for the torsion functional with exponent $1/(n+2)$. 
    
\section{Preliminaries}\label{Sec2}
    \subsection*{The Gauss space}
        Consider the Gaussian measure
        \begin{equation*}
            \gamma(\Omega)=(2\pi)^{-\frac{n}{2}}\int_\Omega e^{-\frac{|x|^2}{2}}\,dx,
        \end{equation*}
        where $\Omega\subset\mathbb{R}^n$, $n\geq1$, is Lebesgue measurable. The Gauss space is the space $\mathbb{R}^n$ endowed with the measure $\gamma$. We observe that $(\mathbb{R}^n,\gamma)$ is a probability space. Note that $\gamma$ is not invariant under translation, while it is invariant under rotations about the origin. We mention that, given a measurable set $\Omega$, it is possible to define the Gaussian perimeter of $\Omega$ in $\mathbb{R}^n$. Then, the Gaussian isoperimetric inequality (we refer to \cite[\S 5.7]{F} for the precise definition of the Gaussian perimeter and for a comprehensive discussion on the inequality) states that, among all subsets of $\mathbb{R}^n$ with prescribed Gaussian measure, half-spaces have the least Gaussian perimeter. We will return in the next two subsections to the role of half-spaces as isoperimetric figures. In what follows, given an open, bounded set $\Omega\subset\mathbb{R}^n$, $W^{1,2}(\Omega,\gamma)$ will denote the weighted Sobolev space of $L^2$ functions with weak first derivative in $L^2$. Then, $W_0^{1,2}(\Omega,\gamma)$ will be the space of $W^{1,2}(\Omega,\gamma)$ functions with zero-trace.
        
    \subsection*{The first eigenvalue of the Ornstein-Uhlenbeck operator}
        Let $\Omega\subset\mathbb{R}^n$ be an open, bounded set with Lipschitz boundary. The Gaussian version of the principal frequency can be defined as
        \begin{equation}\label{infEig}
            \lambda_\gamma(\Omega)=\inf_{w\in W^{1,2}_0(\Omega,\gamma)\setminus\{0\}}R_\gamma(w), 
        \end{equation}
        where
        \begin{equation*}
            R_\gamma(w)=\frac{\int_\Omega |\nabla w|^2\,d\gamma}{\int_\Omega w^2 \,d\gamma}.
        \end{equation*}

        The next result shows that \eqref{infEig} is actually a minimum.
        \begin{teorema}
            Let $\Omega\subset\mathbb{R}^n$ be an open, bounded set with Lipschitz boundary. Then there exists $v\in W^{1,2}_0(\Omega,\gamma)\setminus\{0\}$ such that $\lambda_\gamma(\Omega)=R_\gamma(v).$
        \end{teorema}
        The following result characterizes \eqref{infEig} as a solution to a partial differential equation (see e.g. \cite[Proposition~2.10]{HL}). 
        \begin{teorema}
            Let $\Omega\subset\mathbb{R}^n$ be an open, bounded set with Lipschitz boundary and let $v\in W^{1,2}_0(\Omega,\gamma)\setminus\{0\}$ be a minimizer in \eqref{infEig}. Then $v$ satisfies
            \begin{equation}\label{Probeigen}
                \begin{cases}
                    -\mathcal{L_\gamma}v=\lambda_\gamma(\Omega)v & \text{ in }\Omega\\
                    v=0 & \text{ on } \partial\Omega.
                \end{cases}
            \end{equation}
        \end{teorema}
        \begin{oss}
            Since $R_\gamma(|v|)=R_\gamma(v)$, we can choose a nonnegative minimizer $v\geq0$. Furthermore, from \eqref{Probeigen} and the maximum principle, we can conclude that $v>0$ in $\Omega$ and $v=0$ on $\partial\Omega$.      
        \end{oss}
        We observe that the solution to \eqref{Probeigen} is unique up to a positive multiplicative constant, that is, the first eigenvalue $\lambda_\gamma$ is simple. Finally, the definition of the eigenvalue can be extended to unbounded domains. We refer, for instance, to \cite{E}, where in particular it is proven that half-spaces minimize the first eigenvalue among sets of given Gaussian measure, i. e.,  for every $\Omega\subset\mathbb{R}^n$ open
        \begin{equation*}
            \lambda_\gamma(\Omega)\geq\lambda_\gamma(H),
        \end{equation*}
        where $H$ is a half-space such that $\gamma(\Omega)=\gamma(H)$ (see also \cite{BCF} for more details on the topic, \cite{CCLMP} for a quantitative improvement and \cite{CG} for the extension to Robin boundary condition). This result represents the Gaussian counterpart of the Faber-Krahn inequality for the first eigenvalue of the Laplacian.
    \subsection*{The torsional rigidity for the Ornstein-Uhlenbeck operator}
        Let $\Omega\subset\mathbb{R}^n$ be an open, bounded set with Lipschitz boundary. The Gaussian torsional rigidity can be defined as
        \begin{equation*}
            T_\gamma(\Omega)=\sup_{w\in W^{1,2}_0(\Omega,\gamma)}\frac{\left(\int_\Omega w\,d\gamma\right)^2}{\int_\Omega
             |\nabla w|^2\,d\gamma}.
        \end{equation*}
        As in the case of the first eigenvalue, it is possible to show that this supremum is actually a maximum, attained by the solution of
        \begin{equation}\label{probTor}
            \begin{cases}
                -\mathcal{L}_\gamma v=1 & \text{ in } \Omega\\
                v=0 & \text{ on } \partial\Omega.
            \end{cases}
        \end{equation}
        Moreover, it is not difficult to check that 
        \begin{equation}\label{e:TorFormula}
            T_\gamma(\Omega) = \int_\Omega v(x) \,d\gamma(x),
        \end{equation}
        where $v$ is the solution to \eqref{probTor}. 
        
        We observe that the definition of Gaussian torsional rigidity can be extended to unbounded domains. We refer, again, to \cite{E}, where it is proved that half-spaces maximize the Gaussian torsional rigidity among sets of given Gaussian measure. Specifically,  for every $\Omega\subset\mathbb{R}^n$ open
        \begin{equation*}
            T_\gamma(\Omega)\leq T_\gamma(H),
        \end{equation*}
        where $H$ is a half-space such that $\gamma(\Omega)=\gamma(H)$ (see also \cite{CGNT} for the extension to Robin boundary condition). This result represents the Gaussian counterpart of the Saint-Venant inequality for the torsional rigidity of the Laplacian. We mention another isoperimetric result concerning the first eigenvalue and torsional rigidity related to the Ornstein-Uhlenbeck operator. In more detail, in \cite{HL}, the authors prove that the half-space minimizes the first eigenvalue among sets with given torsional rigidity, that is
        \begin{equation*}
            \lambda_\gamma(\Omega)\geq\lambda_\gamma(H),
        \end{equation*}
        where $H$ is a half-space such that $T_\gamma(\Omega)=T_\gamma(H)$, for every open, bounded subset $\Omega$  of $\mathbb{R}^n$ with Lipschitz boundary. This result represents the Gaussian counterpart of the Kohler-Jobin inequality (see \cite{B} for a complete discussion on the topic).
        
\section{Counterexamples}\label{Counterexample}
\subsection*{Torsional rigidity}
 We start by solving the general $1$-dimensional problem, after which we will exhibit the particular choice of $\Omega_i$. Let $\Omega_i=[a_i,b_i]$, then the torsion problem \eqref{probTor} becomes
        \begin{equation*}
            \begin{cases}
                u''(x)-xu'(x)=-1 & a_i<x<b_i\\
                u(a_i)=u(b_i)=0.
            \end{cases}
        \end{equation*}
        Multiplying in the equation by $e^{-\frac{x^2}{2}}$ and integrating twice we find
        \begin{equation}\label{sol1D}
            u(x)=\int_{a_i}^xe^\frac{t^2}{2}\left(C-\int_{a_i}^t e^{-\frac{s^2}{2}}\,ds \right)\,dt.
        \end{equation}
        From the boundary condition $u(b_i)=0$ the constant is given by
        \begin{equation*}
            C=\frac{\int_{a_i}^{b_i}e^\frac{t^2}{2}\int_{a_i}^te^{-\frac{s^2}{2}}\,ds\,dt}{\int_{a_i}^{b_i}e^\frac{t^2}{2}\,dt}.
        \end{equation*}
        Then, from \eqref{sol1D}, we have
        \begin{equation}\label{Tors1D}
            T_\gamma([a_i,b_i])=\frac{1}{\sqrt{2\pi}}\int_{a_i}^{b_i}u(x)e^{-\frac{x^2}{2}}\,dx=\frac{1}{\sqrt{2\pi}}\int_{a_i}^{b_i}\left[\int_{a_i}^xe^\frac{t^2}{2}\left(C-\int_{a_i}^x e^{-\frac{s^2}{2}}\,ds \right)\,dt \right]e^{-\frac{x^2}{2}}\,dx.
        \end{equation}
        From the statement of Theorem \ref{Teo1}, one may expect that the result also holds true for non-centrally symmetric sets. To highlight the role of central symmetry, we will provide a $1$-dimensional (non-centrally symmetric) example in which $T_\gamma(\Omega_t)>(1-t)T_\gamma(\Omega_0)+tT_\gamma(\Omega_1)$ for some $t\in (0,1)$.
        \begin{esempio}
             Choosing $\Omega_0=[-1,1]$, $\Omega_1=[-2,0]$ and $t=1/2$. Then, from \eqref{Tors1D} we find that 
             $$
             T_\gamma(\Omega_t)-(1-t)T_\gamma(\Omega_0)-tT_\gamma(\Omega_1)\approx0.024229258576.
             $$
        \end{esempio}
        In addition, we will show in the following example that neither log-concavity (which we cannot exclude a priori) nor log-convexity holds in general for centrally symmetric sets:
         \begin{esempio}
           Let $\Omega_0=[-1,1]$, $\Omega_1=[-3,3]$ and $t=3/4$. Then, from \eqref{Tors1D} we find
            \begin{equation*}
                T_\gamma(\Omega_t)-T_\gamma(\Omega_0)^{1-t}T_\gamma(\Omega_1)^t\approx-0.296068463159.
            \end{equation*}
            Moreover, choosing $\Omega_0=[-\frac{1}{2},\frac{1}{2}]$, $\Omega_1=[-1,1]$ and $t=1/2$. Then, from \eqref{Tors1D} we find
            \begin{equation*}
                T_\gamma(\Omega_t)-T_\gamma(\Omega_0)^{1-t}T_\gamma(\Omega_1)^t\approx0.0184327921054.
            \end{equation*}
        \end{esempio}

     We will now show that neither concavity nor convexity of the functional $T_\gamma^\frac{1}{n+2}$ can be proven,  which will be the content of Example \ref{Example-Tor-Exp}. In this regard, we first need to introduce a formula for $T_\gamma(B_R)$; let $u$ be the solution of
        \begin{equation}\label{ProbBall}
            \begin{cases}
                -\mathcal{L}_\gamma u=1 & \text{ in } B_{R},\\
                u=0 & \text{ on } \partial B_{R}.
            \end{cases}
        \end{equation}
        Since the domain in which we are solving the equation is a ball centered at the origin, we know that the solution is a radial function. Using
        \begin{equation*}
            \Delta=\partial_{rr}+\frac{n-1}{r}\partial_r+\frac{1}{r^2}\Delta_{\mathbb{S}^{n-1}},\quad \esc{x,\nabla}=r\partial_r,
        \end{equation*}
        where $\Delta_{\mathbb{S}^{n-1}}$ is the \emph{Laplace-Beltrami operator}, the PDE in \eqref{ProbBall} becomes the following ODE
        \begin{equation*}
            \begin{cases}
                u''(r)+\left(\frac{n-1}{r}-r\right)u'(r)=-1 & 0<r<R\\
                u'(0)=u(R)=0.
            \end{cases}
        \end{equation*}
        Then, we find
        \begin{equation*}
            u(r)=\int_r^R e^\frac{t^2}{2}t^{-(n-1)}\left(\int_0^t s^{n-1}e^{-\frac{s^2}{2}}\,ds\right)\,dt.
        \end{equation*}
        Now, using \eqref{e:TorFormula}, we can see the torsion functional as the following function:
        \begin{equation}\label{T}
            T_\gamma(R):=T_\gamma(B_R)=n\omega_n(2\pi)^{-\frac{n}{2}}\int_0^R\frac{\left(\int_0^t s^{n-1}e^{-\frac{s^2}{2}}\,ds\right)^2}{t^{n-1}e^{-\frac{t^2}{2}}}\,dt.
        \end{equation}
        We need to compute the first and second derivative of \eqref{T}. First, a straightforward computation shows that
        \begin{equation}\label{T'}
            T_\gamma'(R)=n\omega_n(2\pi)^{-\frac{n}{2}}\frac{\left(\int_0^R s^{n-1}e^{-\frac{s^2}{2}}\,ds\right)^2}{R^{n-1}e^{-\frac{R^2}{2}}}.
        \end{equation}
        Thus, differentiating again with respect to $R$, we find
        \begin{equation}\label{T''}
            T_\gamma''(R)=n\omega_n(2\pi)^{-\frac{n}{2}}\int_0^R s^{n-1}e^{-\frac{s^2}{2}}\,ds\left(2-\frac{n-1-R^2}{R^ne^{-\frac{R^2}{2}}}\int_0^Rs^{n-1}e^{-\frac{s^2}{2}}\,ds \right).
        \end{equation}
        
    \begin{esempio}\label{Example-Tor-Exp}
        Let $G_\gamma(R)=T_\gamma(R)^\frac{1}{n+2}$. A straightforward computation gives
        \begin{equation*}
            G_\gamma''(R)=-\frac{n+1}{(n+2)^2}T_\gamma(R)^{-\frac{2n+3}{n+2}}T_\gamma'(R)^2+\frac{1}{n+2}T_\gamma(R)^{-\frac{n+1}{n+2}}T_\gamma''(R).
        \end{equation*}
        Therefore, the statement that $G_\gamma$ is convex is equivalent to prove that $G_\gamma''(R)>0$ and the last, since $T_\gamma(R)>0$ for all $R>0$, is equivalent to prove 
        \begin{equation*}
            T_\gamma''(R)-\frac{n+1}{n+2}\frac{T_\gamma'(R)^2}{T_\gamma(R)}>0.
        \end{equation*}
        Then, from \eqref{T}-\eqref{T'}-\eqref{T''}, the convexity (concavity) of $G_\gamma(R)$ is equivalent to the positivity (negativity) of the following function:
        \begin{equation*}
            f(R)=2-\frac{n-1-R^2}{R^ne^{-\frac{R^2}{2}}}\int_0^R s^{n-1}e^{-\frac{s^2}{2}}\,ds-\frac{n+1}{n+2}\frac{\left(\int_0^R s^{n-1}e^{-\frac{s^2}{2}}\,ds \right)^3}{R^{2n-2}e^{-R^2}}\left(\int_0^R\frac{\left(\int_0^t s^{n-1}e^{-\frac{s^2}{2}}\,ds \right)^2}{t^{n-1}e^{-\frac{t^2}{2}}}\,dt\right)^{-1}.
        \end{equation*}
        Considering, for instance, $n=3$, one has $f(1)\approx-0.019$, while $f(2)\approx 0.248$. This phenomenon is not specific to $n=3$ but seems to occur for every $n\geq 3$. In the appendix, along with numerical evidence for $n=1,2$, plots for $n\geq3$ are included, and it will be possible to see explicitly what has been stated.
    \end{esempio}
\subsection*{First eigenvalue}
    Before giving the counterexample, we need an expression of the eigenvalue on a ball centered at the origin. As in the case of the Laplacian, this will be possible by looking at the first eigenvalue as a zero of a special function. Consider the eigenvalue problem \eqref{Probeigen} where $\Omega=B_R$. We can write the differential operator as follows
    \begin{equation*}
        \Delta-\esc{x,\nabla}=\partial_{rr}+\frac{n-1}{r}\partial_r+\frac{1}{r^2}\Delta_{\mathbb{S}^{n-1}}-r\partial_r.
    \end{equation*}
    Now we expand $u$ in terms of spherical harmonics
    \begin{equation*}
        u(r,\theta)=\sum_{l\geq0}R_l(r)Y_l(\theta),
    \end{equation*}
    where $\Delta_{\mathbb{S}^{n-1}}Y_l(\theta)=-l(l+n-2)Y_l(\theta)$. Then the equation becomes
    \begin{equation}\label{ODEeigen}
        R_l''(r)+\left(\frac{n-1}{r}-r\right)R_l'(r)-\frac{l(l+n-2)}{r^2}R_l(r)=-\lambda_\gamma(\Omega)R_l(r)
    \end{equation}
    for all $l\geq 0$. Introducing the change of variable $r=\sqrt{2s}$, we have
    \begin{equation}\label{chainRule}
        \frac{d}{dr}=\sqrt{2s}\frac{d}{ds},\quad \frac{d^2}{dr^2}=2s\frac{d^2}{ds^2}+\frac{d}{ds}.
    \end{equation}
    Then, if we consider $R_l(r)=s^\frac{l}{2}w(s)$, applying \eqref{chainRule} we find
    \begin{equation*}
        \begin{split}
            R_l'(r)&=\frac{l}{\sqrt{2}}s^\frac{l-1}{2}w(s)+\sqrt{2}s^lw'(s),\\
            R_l''(r)&=\frac{l(l-1)}{2}s^\frac{l-2}{2}w(s)+(2l+1)s^\frac{l}{2}w'(s)+2s^\frac{l+2}{2}w''(s).
        \end{split}
    \end{equation*}
    With this notation, the ODE \eqref{ODEeigen} becomes 
    \begin{equation*}
        sw''(s)+\left(l+\frac{n}{2}-s\right)w'(s)-\frac{l-\lambda_\gamma(\Omega)}{2}w(s)=0,
    \end{equation*}
    which is in the form of Kummer's equation
    \begin{equation*}
        sw''(s)+(b-s)w'(s)-aw(s)=0
    \end{equation*}
    with $b=l+\frac{n}{2}$ and $a=\frac{l-\lambda_\gamma(\Omega)}{2}$. A solution $w(s)$ will be a linear combination of the following functions
    \begin{itemize}
        \item \emph{Kummer confluent hypergeometric function}: $M\left(\frac{l-\lambda_\gamma(\Omega)}{2},\, l+\frac{n}{2};\, s\right)=\,_1F_1\left(\frac{l-\lambda_\gamma(\Omega)}{2};\, l+\frac{n}{2};\, s\right)$,
        \item \emph{Tricomi confluent hypergeometric function}: $U\left(\frac{l-\lambda_\gamma(\Omega)}{2},\, l+\frac{n}{2};\, s\right)$,
    \end{itemize}
    see \cite[Chapter 13]{AS} for a more comprehensive discussion. Since we are interested in a solution that is regular at $r=0$, we have
    \begin{equation*}
        R_l(r)=r^l\,_1F_1\left(\frac{l-\lambda_\gamma(\Omega)}{2};\, l+\frac{n}{2};\, \frac{r^2}{2}\right)
    \end{equation*}
    and from the boundary condition $R_l(R)=0$ we get
    \begin{equation*}
        _1F_1\left(\frac{l-\lambda_\gamma(\Omega)}{2};\, l+\frac{n}{2};\, \frac{R^2}{2}\right)=0.
    \end{equation*}
    Thus, the first eigenvalue is given by $R_0(R)=0$, that is
    \begin{equation}\label{FirstEigen}
        _1F_1\left(-\frac{\lambda_\gamma(\Omega)}{2};\, \frac{n}{2};\, \frac{R^2}{2}\right)=0.
    \end{equation}
    We are now ready for the counterexample.
    \begin{esempio}
        Consider $R_0=4$, $R_1=6$, $n=2$, $t=1/2$. Then $R_t=(1-t)R_0+tR_1=5$. Using \eqref{FirstEigen}, we can approximate numerically the following eigenvalues:
        \begin{equation}
            \begin{split}
                \lambda_\gamma(B_{R_0})&\approx 0.0045931899,\\
                \lambda_\gamma(B_{R_1})&\approx 0.0000005157,\\
                \lambda_\gamma(B_{R_t})&\approx0.0000848928.
            \end{split}
        \end{equation}
        Then,
        \begin{equation*}
            \lambda_\gamma(B_{R_t})^{-\frac{1}{2}}-\frac{1}{2}\lambda_\gamma(B_{R_0})^{-\frac{1}{2}}-\frac{1}{2}\lambda_\gamma(B_{R_1})^{-\frac{1}{2}}\approx-595.
        \end{equation*}
    \end{esempio}

\section{Main result}\label{Proof of the results}
The present section deals with the proof of Theorem \ref{Teo1}. To make the explanation clearer, we will now present an outline of the proof, which we will formalize in the course of this section.

We want to prove that the function $f(R)=T_\gamma(R)^\alpha$ is convex for some $\alpha\in(0,1)$, that is, to prove that $f''(R)>0$ for all $R>0$, and that the best possible choice of $\alpha$, independent of the dimension, is $\alpha=1/3$. A straightforward computation shows that $f''(R)>0$ is equivalent to
\begin{equation*}
    \alpha\geq 1-\frac{T_\gamma(R)T_\gamma''(R)}{T_\gamma'(R)^2}=:\alpha_n(R).
\end{equation*}
The claim is then
\begin{equation*}
    \sup_{n\geq1}\sup_{R>0}\alpha_n(R)=\frac{1}{3}.
\end{equation*}
To do this, we will get the following information about $\alpha_n(R)$:
\begin{itemize}
    \item [(i)] $\lim_{R\rightarrow0^+}\alpha_n(R)=\frac{1}{n+2}\leq\frac{1}{3}$;
    \item [(ii)] $\lim_{R\rightarrow+\infty}\alpha_n(R)=0$;
    \item[(iii)] whenever 
    \begin{equation*}
        \frac{d}{dR}\alpha_n(R)\bigg|_{R=R^*}=0,
    \end{equation*}
    for some $R^*>0$, then $\alpha_n(R^*)\leq1/3$.
\end{itemize}
In more detail, (i) implies that there exists $\delta > 0$ such that $\alpha_n(R)<1/3$ for all $R\in(0,\delta)$. In a similar manner, (ii) implies the existence of $M >0$ such that $\alpha_n(R)<1/3$ for all $R>M$.

Once we know this, if for some $R_0>0$ and $n\geq1$ we have $\alpha_n(R_0)>1/3$, then from what we have said it must be $R_0\in[\delta ,M]$. Then, from the Weierstrass theorem, there exists $R^*\in[\delta,M]$ such that
\begin{equation*}
    \alpha_n(R^*)\geq \alpha_n(R_0)>\frac{1}{3},\quad \frac{d}{dR}\alpha_n(R)\bigg|_{R=R^*}=0,
\end{equation*}
which is a contradiction to (iii). 

The proof of Theorem \ref{Teo1} is organized as follows. We first prove points (i)-(ii)-(iii), which will be the content of the next four lemmas. The conclusion is then achieved in the proof of Theorem \ref{Teo1}. We start by introducing the following notation:
\begin{equation}\label{newdef}
    \begin{split}
        &c_n=n\omega_n(2\pi)^{-\frac{n}{2}},\quad I(R)=\int_0^R s^{n-1}e^{-\frac{s^2}{2}}\,ds,\quad J(R)=\int_0^R\frac{I(t)^2}{I'(t)}\,dt,\\
        &h(R)=\frac{I(R)}{I'(R)},\quad v(R)=\frac{J(R)}{h(R)I(R)},\quad p(R)=\frac{n-1}{R}-R.
    \end{split}
\end{equation}
\begin{lemma}\label{Lemma1}
    Let $v$ be defined as in \eqref{newdef}. Then
    \begin{equation*}
        \lim_{R\rightarrow 0^+}v'(R)=\frac{1}{n+2}.
    \end{equation*}
\end{lemma}
\begin{proof}
    The idea is to write $v(R)$ for $R\sim0$. Recalling the definition of $v$ from \eqref{newdef}, we start by expanding $e^{-\frac{s^2}{2}}$ under the integral sign
    \begin{equation*}
        I(R)=\int_0^R s^{n-1}\left(1-\frac{s^2}{2}+O(s^4)\right)\,ds=\frac{R^n}{n}-\frac{R^{n+2}}{2(n+2)}+O(R^{n+4}),
    \end{equation*}
    and
    \begin{equation*}
        I'(R)=R^{n-1}\left(1-\frac{R^2}{2}+O(R^4)\right).
    \end{equation*}
    From these, we can write
    \begin{equation*}
        h(R)=\frac{I(R)}{I'(R)}=\frac{R}{n}+O(R^3).
    \end{equation*}
    Hence,
    \begin{equation}\label{exphI}
        h(R)I(R)=\frac{R^{n+1}}{n^2}+O(R^{n+3}).
    \end{equation}
    
    Regarding $J(R)$, since the integrand satisfies
    \begin{equation*}
        \frac{I(t)^2}{I'(t)}=\frac{t^{n+1}}{n^2}+O(t^{n+3}),
    \end{equation*}
    we get
    \begin{equation}\label{expJ}
      J(R)=\frac{R^{n+2}}{n^2(n+2)}+O(R^{n+4}).  
    \end{equation}
    Then, combining \eqref{exphI}-\eqref{expJ}, we find
    \begin{equation*}
        v(R)=\frac{J(R)}{h(R)I(R)}=\frac{R}{n+2}+O(R^3).
    \end{equation*}
    The latter gives the claim.
\end{proof}
\begin{lemma}\label{Lemma2}
    Let $v$ be defined as in \eqref{newdef}. Then
    \begin{equation*}
        \lim_{R\rightarrow +\infty}v'(R)=0.
    \end{equation*}
\end{lemma}
\begin{proof}
   First, note that
    \begin{equation*}
        \begin{split}
            &I(R)\rightarrow c_n=\int_0^{+\infty}s^{n-1}e^{-\frac{s^2}{2}}\,ds,\\
            &I'(R)=R^{n-1}e^{-\frac{R^2}{2}}\rightarrow0,
        \end{split}
    \end{equation*}
    as $R\rightarrow+\infty$. Hence, $h(R)=I(R)/I'(R)\rightarrow+\infty$ and $h(R)\sim c_nR^{1-n}e^\frac{R^2}{2}$. Now, recalling that
    \begin{equation*}
         p(R)=\frac{n-1}{R}-R=\frac{I''(R)}{I'(R)},
    \end{equation*}
    we find
    \begin{equation}\label{h'(R)}
        h'(R)=1-\frac{I(R)I''(R)}{I'(R)^2}=1-\frac{I(R)p(R)}{I'(R)}=1-h(R)p(R).
    \end{equation}
    Thus, using \eqref{h'(R)} and the definition of $p(R)$, we get that
    \begin{equation}\label{limit1}
        \frac{1+h'(R)}{Rh(R)}=\frac{2-h(R)p(R)}{Rh(R)}=1+\frac{2}{Rh(R)}-\frac{n-1}{R^2}\rightarrow 1,
    \end{equation}
    as $R\rightarrow+\infty$.
    
    Define now $\varphi(R)=Rv(R)$. Starting from the definition of $v(R)$, a straightforward computation gives
    \begin{equation*}
        v'(R)=1-\frac{2v(R)}{h(R)}+p(R)v(R),
    \end{equation*}
    then
    \begin{equation*}
        \varphi'(R)=R+v(R)\left(1+Rp(R)-\frac{2R}{h(R)}\right).
    \end{equation*}
    Moreover, we find, from the definition of $p(R)$, that 
    \begin{equation*}
        1+Rp(R)-\frac{2R}{h(R)}=n-R^2-\frac{2R}{h(R)},
    \end{equation*}
    which, together with $v(R)=\varphi(R)/R$, gives
    \begin{equation*}
        \varphi'(R)=R-R\varphi(R)+\varphi(R)\left(\frac{n}{R}-\frac{2}{h(R)} \right).
    \end{equation*}
    Finally, we get
    \begin{equation}\label{eq(2)}
        \varphi'(R)=R(1-\varphi(R))+E(R)\varphi(R),
    \end{equation}
    where
    \begin{equation*}
        E(R)=\frac{n}{R}-\frac{2}{h(R)},
    \end{equation*}
    and $E(R)\rightarrow0$ as $R\rightarrow+\infty.$
    
    We will now prove that $\varphi$ is bounded. In this regard, we first observe that, since $v(R)>0$ for all $R>0$, then $\varphi(R)>0$. Furthermore, since $E(R)\rightarrow0$, consider $R_0>0$ such that, for instance, $|E(R)|\leq1/2$ for all $R\geq R_0$. Now, if for some $R_1\geq \max\{R_0,\,2\}$, $\varphi(R_1)=2$, then from \eqref{eq(2)} and the boundedness of $E(R)$ we have
    \begin{equation*}
        \varphi'(R_1)=-R_1+2E(R_1)\leq1-R_1<0.
    \end{equation*}
    Therefore, $\varphi$ is pushed below the value $\varphi(R_1)=2$, and we can conclude that $0<\varphi(R)\leq M$, for all $R>0$, where 
    \begin{equation*}
        M=\max\left\{2,\sup_{R\in(0,R_0]}\varphi(R) \right\}.
    \end{equation*}
    
    We will finish the proof by showing that $\varphi(R)\rightarrow 1$ as $R\rightarrow+\infty$, which together with \eqref{limit1} will imply the thesis.
    In this regard, fix $\delta>0$ and consider $R_1>1+M$ large enough such that
    \begin{equation}\label{eq(4)}
        |E(R)|\leq\delta
    \end{equation}
   for all $R\geq R_1$. We first prove that $\varphi$ cannot cross $1+\delta$ from below. To do this, assume by contradiction that it is false, namely that there exists $R_2\geq R_1$ such that
    \begin{equation*}
        \varphi(R_2)=1+\delta\quad\text{ and }\quad \varphi'(R_2)\geq0.
    \end{equation*}
    Then, from \eqref{eq(2)}-\eqref{eq(4)}, we find
    \begin{equation*}
        \varphi'(R_2)=-\delta R_2+E(R_2)(1+\delta)\leq \delta(1+M-R_2)<0,
    \end{equation*}
    where we have used $\delta<M$ (which we can always assume, provided that we choose $\delta$ sufficiently small) and $R_2\geq R_1>1+M$. The latter gives $0\leq \varphi'(R_2)<0$, that is a contradiction. Arguing in the same way, it is possible to show that $\varphi$ cannot cross $1-\delta$ from above, that is, there does not exist $R_2\geq R_1$ such that $\varphi(R_2)=1-\delta$ and $\varphi'(R_2)\leq0$.
    
    To summarize, what we have shown so far is that, if for a certain $\overline{R}\geq R_1$, we have $1-\delta\leq\varphi(\overline{R})\leq1+\delta$, then $1-\delta\leq\varphi(R)\leq1+\delta$ for all $R\geq \overline{R}$. Assume now that $\varphi(R)>1+\delta$ for all $R\geq R_1$. Then, from \eqref{eq(2)}-\eqref{eq(4)}, we have
    \begin{equation*}
        \varphi'(R)=R(1-\varphi(R))+E(R)\varphi(R)\leq -\delta R+\delta\varphi(R)\leq -\delta R+\delta(1+M)<0,
    \end{equation*}
    where we have used again $R\geq R_1>1+M$ and $\varphi(R)\leq M<1+M$. In particular, if we integrate in the previous inequality on $[R_1,R]$, we find
    \begin{equation*}
        \varphi(R)\leq \varphi(R_1)+(R-R_1)\left(\delta(1+M)-\frac{\delta}{2}(R+R_1)\right)\rightarrow-\infty
    \end{equation*}
    as $R\rightarrow+\infty$. This is a contradiction since $\varphi(R)>0$. This ensures that there exists $\overline{R}_1\geq R_1$ such that $\varphi(\overline{R}_1)\leq 1+\delta$ and, from what we have proved before, $\varphi(R)\leq 1+\delta$ for all $R\geq\overline{R}_1$.
    Assume now $\varphi(R)<1-\delta$ for all $R\geq R_1$. Hence, from \eqref{eq(2)}, we find
    \begin{equation*}
        \begin{split}
            \varphi'(R)&=R(1-\varphi(R))+E(R)\varphi(R)>\delta R+E(R)\varphi(R)>\delta R-\delta\varphi(R)\\
            &>\delta R-\delta(1-\delta).
        \end{split} 
    \end{equation*}
    Integrating again on $[R_1,R]$ the previous inequality, we find
    \begin{equation*}
        \varphi(R)\geq \varphi(R_1)+(R-R_1)\left[ \frac{\delta}{2}(R+R_1)-\delta(1-\delta)\right]\rightarrow+\infty,
    \end{equation*}
    which contradicts the fact $\varphi(R)\leq M$. Thus, there exists $\overline{R}_2\geq R_1$ such that $\varphi(\overline{R}_2)\geq 1-\delta$ and, from what we have proved, $\varphi(R)\geq1-\delta$ for all $R\geq \overline{R}_2$. Finally, we have found that, for all $R\geq\max\{\overline{R}_1,\,\overline{R}_2\}$, it holds
    \begin{equation*}
        1-\delta\leq \varphi(R)\leq 1+\delta.
    \end{equation*}
    
    Since $\delta$ was arbitrary, we can conclude that $\varphi(R)\rightarrow1$ as $R\rightarrow+\infty$. To finish the proof, observe that the latter and \eqref{limit1} imply that
    \begin{equation*}
        v'(R)=1-\frac{1+h'(R)}{h(R)}v(R)=1-\frac{1+h'(R)}{Rh(R)}\varphi(R)\rightarrow1-1=0
    \end{equation*}
    as $R\rightarrow+\infty$.
\end{proof}
\begin{lemma}\label{Lemma3}
    Let $h$ be defined as in \eqref{newdef}, then $h'(R)>0$. 
\end{lemma}
\begin{proof}
    We first recall that, from \eqref{h'(R)}, we have $h'(R)=1-h(R)p(R)$. We distinguish three cases.
    \begin{enumerate}[label=\textbf{\underline{Case \arabic*:}}, wide]
    \item $n=1$. In this case we have $p(R)=-R<0$, while $h(R)>0$ for all $R>0$, then $h'(R)=1-h(R)p(R)>0$ for all $R>0$.

    \item $n\geq2$ and $R\geq\sqrt{n-1}$. In this case we have
    \begin{equation*}
        p(R)=\frac{n-1}{R}-R\leq\sqrt{n-1}-R\leq0,
    \end{equation*}
    while $h(R)>0$ for all $R>0$, then $h'(R)=1-h(R)p(R)>0$ for all $R>0$.

    \item $n\geq2$ and $0<R<\sqrt{n-1}$. Now we have $p(R),\,h(R)>0$ for all $R>0$.  We first remark that, looking at the expansion of $h(R)$, for $R\sim0$, in the proof of the previous lemma, we have
    \begin{equation*}
        \lim_{R\rightarrow0^+}h'(R)=\lim_{R\rightarrow0^+}(1-h(R)p(R))=\lim_{R\rightarrow0^+}\left(1-\left(\frac{n-1}{R}-R\right)\frac{R}{n}\right)=1-\frac{n-1}{n}>0.
    \end{equation*}
    Assume that there exists $R_0\in(0,\sqrt{n-1})$ such that
    \begin{equation*}
        h'(R)>0\quad\forall\, R\in(0,R_0) \quad\text{ and }\quad h'(R_0)=0.
    \end{equation*}
    Thus $h''(R_0)\leq0$, and the latter gives us
    \begin{equation*}
        0\geq h''(R_0)=-h'(R_0)p(R_0)-h(R_0)p'(R_0)=-h(R_0)p'(R_0),
    \end{equation*}
    where in the last equality we have used $h'(R_0)=0$. Then, we find $h(R_0)p'(R_0)\geq0$, that is a contradiction since 
    \begin{equation}\label{p'}
        p'(R)=-\frac{n-1}{R^2}-1<0,
    \end{equation}
    and $h(R)>0$.
\end{enumerate}
   
\end{proof}
\begin{lemma}\label{Lemma4}
    Assume that for some $R^*>0$, it holds $v''(R^*)=0$, then 
    \begin{equation*}
        v'(R^*)=\frac{2h'(R^*)+p'(R^*)h(R^*)^2}{(1+h'(R^*))^2+2h'(R^*)+p'(R^*)h(R^*)^2}\leq\frac{1}{3}.
    \end{equation*}
\end{lemma}
\begin{proof}
    We start by recalling that 
    \begin{equation*}
        v(R)=\frac{J(R)}{h(R)I(R)}.
    \end{equation*}
    Now, differentiating, since $1+h'(R)=2-p(R)h(R)$, we find
    \begin{equation}\label{v'}
        v'(R)=\frac{J'(R)}{h(R)I(R)}-\frac{J(R)(h(R)I(R))'}{h(R)^2I(R)^2}=1-\frac{(1+h'(R))v(R)}{h(R)}=1-\frac{2v(R)}{h(R)}+p(R)v(R).
    \end{equation}
    Differentiating again in \eqref{v'} we have
    \begin{equation}\label{v''}
        v''(R)=v'(R)\left(p(R)-\frac{2}{h(R)}\right)+v(R)\left(\frac{2h'(R)}{h(R)^2}+p'(R)\right).
    \end{equation}
    Assume that there exists $R^*>0$ such that $v''(R^*)=0$. Taking into account that $p(R)h(R)=1-h'(R)$, we get from \eqref{v''} that
    \begin{equation}\label{v'(R^*)}
        v'(R^*)=\frac{v(R^*)\left(\frac{2h'(R^*)}{h(R^*)}+p'(R^*)h(R^*)\right)}{1+h'(R^*)}.
    \end{equation}
    Setting
    \begin{equation*}
        w(R)=\frac{v(R)}{h(R)},\quad A(R)=2h'(R)+p'(R)h(R)^2,\quad s(R)=1+h'(R),
    \end{equation*}
    \eqref{v'(R^*)} becomes 
    \begin{equation*}
        v'(R^*)=\frac{w(R^*)A(R^*)}{s(R^*)}.
    \end{equation*}
    Moreover, from \eqref{v'}, we have $v'(R)=1-s(R)w(R)$. Combining the previous two identities we get
    \begin{equation}\label{w(R^*)}
        w(R^*)=\frac{s(R^*)}{A(R^*)+s(R^*)^2}.
    \end{equation}
    Therefore
    \begin{equation}\label{V'(R^*)2}
        v'(R^*)=1-s(R^*)w(R^*)=\frac{A(R^*)}{A(R^*)+s(R^*)^2}.
    \end{equation}
    By Lemma \ref{Lemma3}, we have $s(R)=1+h'(R)>0$, and $w(R)>0$. Since \eqref{w(R^*)} gives $s(R^*)=w(R^*)(s(R^*)^2+A(R^*))$, then it must be $s(R^*)^2+A(R^*)>0$ and \eqref{V'(R^*)2} is well defined. From \eqref{V'(R^*)2} follows that $v'(R^*)\leq 1/3$ if and only if $2A(R^*)\leq s(R^*)^2$. By the definition of $A(R)$ and $s(R)$, using \eqref{p'}, we find
    \begin{align*}
        s(R^*)^2-2A(R^*)&=(1+h'(R^*))^2-2(2h'(R^*)+p'(R^*)h(R^*)^2)\\
        &=(h'(R^*)-1)^2-2p'(R^*)h(R^*)^2\geq0,
    \end{align*}
    that gives the claim.
\end{proof}
    We are now ready for the proof of the first result stated in \S \ref{Introduction}.
    \begin{proof}[Proof of Theorem \ref{Teo1}]
        Let $\alpha\in(0,1)$ and consider $f(R)=T_\gamma(R)^\alpha=T_\gamma(B_R)^\alpha$. We want to prove that there exists a threshold $\alpha_n$ such that, if $\alpha\in [\alpha_n,1)$, then $f''(R)>0$. From the definition we have
        \begin{equation*}
            f''(R)=\alpha T_\gamma(R)^{\alpha-1}T_\gamma''(R)+\alpha(\alpha-1)T_\gamma(R)^{\alpha-2}T_\gamma'(R)^2,
        \end{equation*}
        where $T_\gamma$, $T_\gamma'$, $T_\gamma''$ are given in \eqref{T}-\eqref{T'}-\eqref{T''}. Hence, it holds $f''(R)>0$ if and only if
        \begin{equation*}
            T_\gamma(R)T_\gamma''(R)+(\alpha-1)T_\gamma'(R)^2\geq0,
        \end{equation*}
        which is equivalent to 
        \begin{equation*}
            \begin{split}
                \alpha&\geq -\frac{T_\gamma(R)T_\gamma''(R)}{T_\gamma'(R)^2}+1\\
                &=\frac{R^{2n-2}e^{-R^2}}{\left(\int_0^R s^{n-1}e^{-\frac{s^2}{2}}\,ds \right)^3}\int_0^R\frac{\left(\int_0^t s^{n-1}e^{-\frac{s^2}{2}}\,ds \right)^2}{t^{n-1}e^{-\frac{t^2}{2}}}\,dt\left[\frac{n-1-R^2}{R^ne^{-\frac{R^2}{2}}}\int_0^R s^{n-1}e^{-\frac{s^2}{2}}\,ds-2 \right]+1=:\alpha_n(R).
            \end{split}
        \end{equation*}
        In the remaining part of the proof, we shall prove that $\alpha_n(R)\leq \frac{1}{3}$. To do this, we will prove that $\alpha_n(R)=v'(R)$, where $v(R)$ is defined in \eqref{newdef}. Since
        \begin{equation*}
            J'(R)=\frac{I(R)^2}{I'(R)}=h(R)I(R),
        \end{equation*}
        we have that $T'_\gamma(R)=c_nh(R)I(R)$, where all these quantities are defined in \eqref{newdef}. Since 
        \begin{equation*}
            I''(R)=[(n-1)R^{n-2}-R^n]e^{-\frac{R^2}{2}}=p(R)I'(R),
        \end{equation*}
        and $h'(R)=1-p(R)h(R)$, we get
        \begin{equation*}
            T''_\gamma(R)=c_n(h(R)I(R))'=c_nI(R)(1+h'(R)).
        \end{equation*}
        Then
        \begin{equation*}
            \alpha_n(R)=1-\frac{T_\gamma(R)T''_\gamma(R)}{T'_\gamma(R)^2}=1-\frac{J(R)(1+h'(R))}{h(R)^2I(R)}=1-\frac{(1+h'(R))v(R)}{h(R)}.
        \end{equation*}
        Computing $v'(R)$, a straightforward computation shows that $\alpha_n(R)=v'(R)$. Assume now by contradiction that there exists $R_0>0$ and $n\geq1$ such that $\alpha_n(R_0)>1/3$. By Lemma \ref{Lemma1} and the continuity of $\alpha_n(R)$ we get $\alpha_n(R)<1/3$ for all $R\in(0,\delta)$, with $\delta>0$ sufficiently small. In the same way, by Lemma \ref{Lemma2}, we have $\alpha_n(R)<1/3$ for $R>M$, with $M>0$ sufficiently large. Since $\alpha_n(R_0)>1/3$ with $R_0\in[\delta,M]$, by the Weierstrass theorem we have a maximum $R^*\in(\delta,M)$ with 
        \begin{equation*}
            \alpha_n(R^*)>\frac{1}{3},\quad\frac{d}{dR}\alpha_n(R)\bigg|_{R=R^*}=0,
        \end{equation*}
        that is a contradiction from Lemma \ref{Lemma4}.
    \end{proof}
    
    \begin{oss}[Equality case]
        Consider $F(t)=T_\gamma(R(t))^\frac{1}{3}$, where $R(t) = (1-t)R_0 + tR_1$ and $t\in (0,1)$. The statement of the theorem is that $G(t)=F(t)-(1-t)F(0)-tF(1)\leq0$ for any $t\in(0,1)$. Since $G(0)=G(1)=0$ and $G(t)$ is convex, if for some $t_0\in(0,1)$ it holds $G(t_0)=0$, then it must be $G(t)=0$ for all $t\in(0,1)$. Hence, we can assume that $G(t)=0$ for all $t\in(0,1)$. This is equivalent to $F''(t)=0$ for all $t\in(0,1)$, and the latter condition is equivalent to the following one
        \begin{equation}\label{eqcond}
            R'(t)^2\left(T_\gamma(R)T_\gamma''(R)-\frac{2}{3}T_\gamma'(R)^2 \right)=0.
        \end{equation}
        We want to show that the quantity in brackets can never be zero. To do this, recall from the proof of Theorem \ref{Teo1} that $T_\gamma(R)T_\gamma''(R)+(\alpha-1)T_\gamma'(R)^2>0$ for all $\alpha> \alpha_n(R)$. Moreover, the previous inequality is an equality if and only if $\alpha=\alpha_n(R)$. In this case, $\alpha=1/3$ and $\alpha_n(R)<1/3$, where the equality is achieved only for $n=1$ as $R\rightarrow0^+$. Therefore, we get from \eqref{eqcond} that $R'(t)=0$ for all $t\in(0,1)$ which implies that $R_0=R_1$. 
    \end{oss}

\textbf{Funding:}
  The first named author is supported by the grant PID2021-124157NB-I00, funded by MCIN/AEI/10.13039/501100011033/``ERDF A way of making Europe'', as well as by the grant  ``Proyecto financiado por la CARM a través de la convocatoria de Ayudas a proyectos para el desarrollo de investigación científica y técnica por grupos competitivos, incluida en el Programa Regional de Fomento de la Investigación Científica y Técnica (Plan de Actuación 2022) de la Fundación Séneca-Agencia de Ciencia y Tecnología de la Región de Murcia, REF. 21899/PI/22''. The second named author is partially supported by Gruppo Nazionale per l'Analisi Matematica, la Probabilità e le loro Applicazioni (GNAMPA) of  Istituto Nazionale di Alta Matematica (INdAM) and he is partially supported by INdAM
GNAMPA 2026 Project "Processi di diffusione non-lineari: regolarità e classificazione delle soluzioni",  CUP E53C25002010001.
  \\
  
  \textbf{Acknowledgments:}
  This work was started while the second named author was visiting the
  Department of Mathematics at the University of Murcia, under the supervision of Professor María A. Hernández Cifre and Professor Jesús Yepes Nicolás, whom we thank for their valuable advice and support.
  \\
  
  \textbf{Conflict of interest:}
   The authors declare that they have no conflicts of interest with respect to this work.
  


\bibliographystyle{plain}
\bibliography{biblio}
\appendix
\section{Appendix}\label{Appendix}
     We now give the proof of Proposition \ref{TeoHalfSp}. To do this, we need the following elementary lemma. A proof is included for the sake of completeness. 
    \begin{lemma}\label{lemmaTeoHS}
        Let $s\in\mathbb{R}$, then
        \begin{equation*}
            \int_s^{+\infty}e^{-\frac{t^2}{2}}\,dt\geq\frac{s}{s^2+1}e^{-\frac{s^2}{2}}.
        \end{equation*}
    \end{lemma}
    \begin{proof}
        Consider the function
        \begin{equation*}
            f(t)=e^{-\frac{t^2}{2}}\frac{t}{t^2+1}.
        \end{equation*}
        Then the first derivative is given by
        \begin{equation*}
            f'(t)=-e^{-\frac{t^2}{2}}\frac{t^4+2t^2-1}{(t^2+1)}.
        \end{equation*}
        Since
        \begin{equation*}
            \frac{t^4+2t^2-1}{(t^2+1)}=1-\frac{2}{(t^2+1)^2}\leq1,
        \end{equation*}
        we can infer that 
        \begin{equation*}
            -f'(t)=e^{-\frac{t^2}{2}}\frac{t^4+2t^2-1}{(t^2+1)}\leq e^{-\frac{t^2}{2}}
        \end{equation*}
        for all $t\in\mathbb{R}$.
        Integrating the latter we get
        \begin{equation*}
            \int_s^{+\infty}e^{-\frac{t^2}{2}}\,dt\geq\int_s^{+\infty}-f'(t)\,dt=f(s)-\lim_{t\rightarrow+\infty}f(t).
        \end{equation*}
        Since $\lim_{t\rightarrow+\infty}f(t)=0$, the claim is proved.
    \end{proof}
    \begin{proof}[Proof of Proposition \ref{TeoHalfSp}]
        Let $s\in\mathbb{R}$ and consider the half-space
        \begin{equation*}
            H_s=\{x=(x_1,\dots,x_n)\in\mathbb{R}^n\,:\, x_1\geq s\}.
        \end{equation*}
        We recall that (see \cite[\S2.3]{HL}) the Gaussian torsional rigidity of $H_s$ is given by
        \begin{equation*}
            T_\gamma(s):=T_\gamma(H_s)=\frac{1}{\sqrt{2\pi}}\int_s^{+\infty}e^\frac{t^2}{2}\left(\int_t^{+\infty}e^{-\frac{\tau^2}{2}}\,d\tau \right)^2\,dt.
        \end{equation*}
        We thus want to show that, for all $s\in\mathbb{R}$, $T_\gamma''(s)\geq0$. A straightforward computation gives 
        \begin{equation*}
            \begin{split}
                T_\gamma'(s)&=-\frac{1}{\sqrt{2\pi}}e^\frac{s^2}{2}\left(\int_s^{+\infty}e^{-\frac{t^2}{2}}\,dt \right)^2,\\
                T_\gamma''(s)&=-\frac{se^\frac{s^2}{2}}{\sqrt{2\pi}}\left(\int_s^{+\infty}e^{-\frac{t^2}{2}}\,dt \right)^2+\frac{2}{\sqrt{2\pi}}\int_s^{+\infty}e^{-\frac{t^2}{2}}\,dt.
            \end{split}
        \end{equation*}
        Hence, $T_\gamma''(s)\geq0$ if and only if
        \begin{equation*}
            2-se^\frac{s^2}{2}\int_s^{+\infty}e^{-\frac{t^2}{2}}\,dt\geq0.
        \end{equation*}
        The latter is clearly true when $s\leq0$. Then, given
        \begin{equation*}
            f(s)=se^\frac{s^2}{2}\int_s^{+\infty}e^{-\frac{t^2}{2}}\,dt,
        \end{equation*}
        we need to show that $f(s)\leq2$. Indeed, we will prove that $f(s)\leq1$ by using the following two facts:
        \begin{itemize}
            \item $f$ is monotonically increasing,
            \item $f(s)\leq\lim_{s\rightarrow+\infty}f(s)=1.$
        \end{itemize}
        On the one hand, we get from Lemma \ref{lemmaTeoHS}
        \begin{equation*}
            \begin{split}
                f'(s)&=e^\frac{s^2}{2}\int_s^{+\infty}e^{-\frac{t^2}{2}}\,dt+s^2e^\frac{s^2}{2}\int_s^{+\infty}e^{-\frac{t^2}{2}}\,dt-s\\
                &\geq \frac{s}{1+s^2} + \frac{s^3}{1+s^2}-s=0.
            \end{split}
        \end{equation*}
        On the other hand, using for instance L’H\^opital’s rule 
        \begin{equation*}
            \lim_{s\rightarrow+\infty}f(s)=\lim_{s\rightarrow+\infty}\frac{-e^{-\frac{s^2}{2}}}{-e^\frac{s^2}{2}-s^2e^\frac{s^2}{2}}s^2e^{s^2}=\lim_{s\rightarrow+\infty}\frac{s^2}{1+s^2}=1.
        \end{equation*}
        as we wanted to see.
    \end{proof}
    
    We conclude with the following discussion: recall from Example \ref{Example-Tor-Exp}, that the sign of $G_\gamma''(R)$, where $G_\gamma(R)=T_\gamma(R)^\frac{1}{n+2}$, is characterized by the sign of the following function
    \begin{equation*}
        f(n,R)=2-\frac{n-1-R^2}{R^ne^{-\frac{R^2}{2}}}\int_0^R s^{n-1}e^{-\frac{s^2}{2}}\,ds-\frac{n+1}{n+2}\frac{\left(\int_0^R s^{n-1}e^{-\frac{s^2}{2}}\,ds \right)^3}{R^{2n-2}e^{-R^2}}\left(\int_0^R\frac{\left(\int_0^t s^{n-1}e^{-\frac{s^2}{2}}\,ds \right)^2}{t^{n-1}e^{-\frac{t^2}{2}}}\,dt\right)^{-1}.
    \end{equation*}
    Since the integrals that appear in the definition of $f(n,R)$ cannot be solved in terms of elementary functions, we limit ourselves to showing the plot of this function for different values of $n$.
    \begin{figure}[H]
        \centering
        \includegraphics[width=7cm]{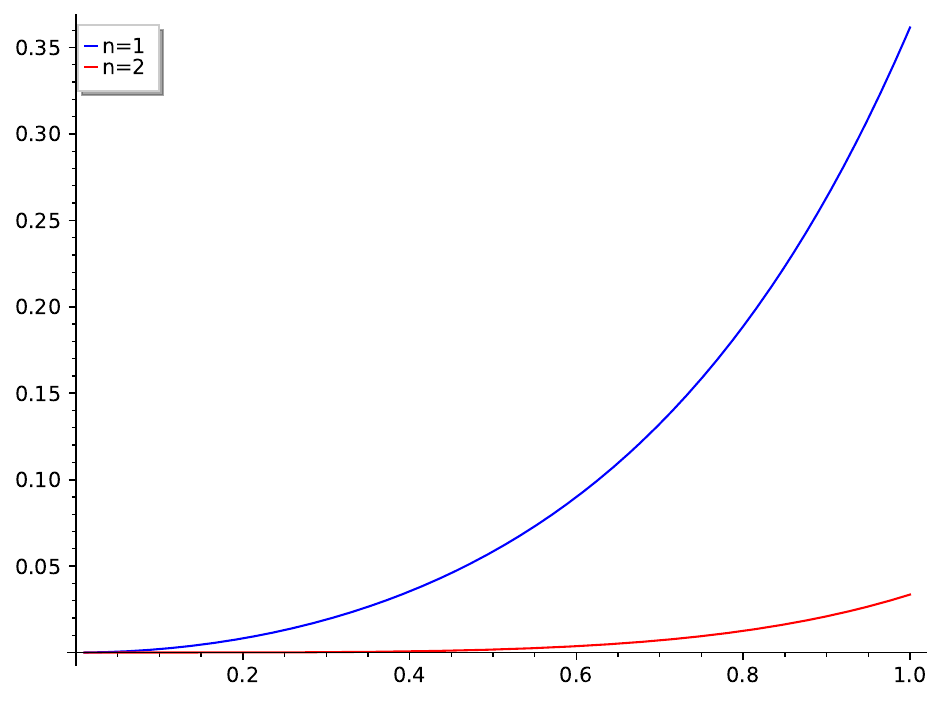}
        \caption{Plot of $f(1,R)$ and $f(2,R)$.}
        \label{Fig1}
    \end{figure}
    As can be seen from Figure \ref{Fig1}, $f(2,R)$ is positive for all $R>0$; this is equivalent to say that
    \begin{equation*}
        T_\gamma(B_{R_t})^\frac{1}{4}\leq (1-t)T_\gamma(B_{R_0})^\frac{1}{4}+tT_\gamma(B_{R_1})^\frac{1}{4},
    \end{equation*}
    for all $R_0,\,R_1>0$, where $R_t=(1-t)R_0+tR_1$. Hence, attending to the case $n=1$, the following question remains open:
    \begin{itemize}
        \item [\textbf{(Q)}]  Let $n\in \{1,2\}$ and $\Omega_0, \Omega_1$ be open, bounded and centrally symmetric subsets of $\mathbb{R}^n$. Is the functional $t \mapsto T_\gamma(\Omega_t)^\frac{1}{n+2}$ convex?
    
    \end{itemize}
    
    The scenario is different when we consider $n\geq3$.
    \begin{figure}[H]
        \centering
        \includegraphics[width=7cm]{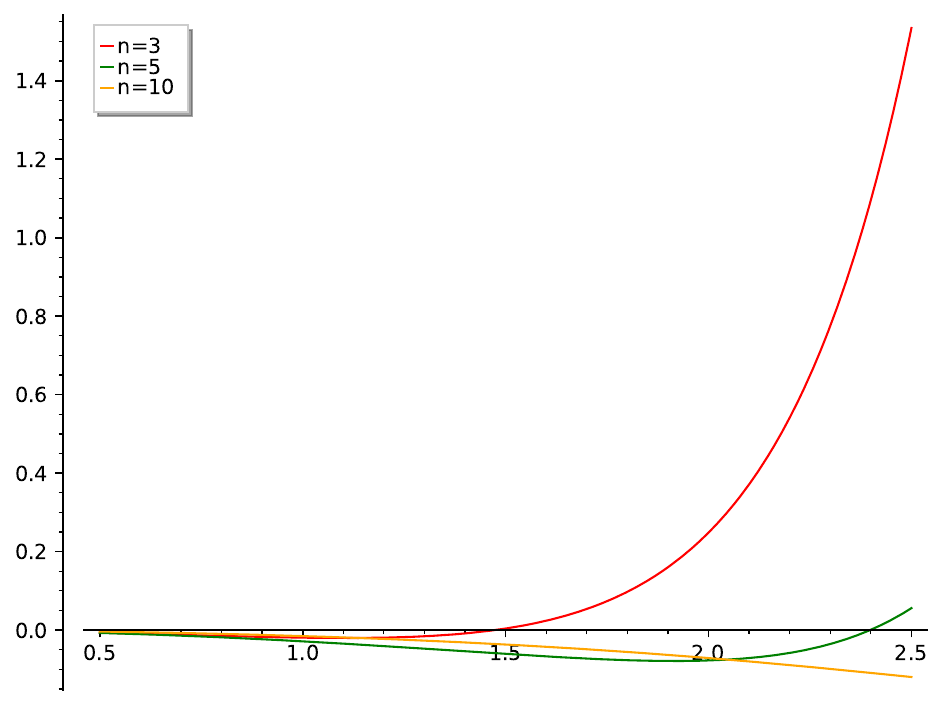}
        \caption{Plot of $f(3,R)$, $f(5,R)$, and $f(10,R)$.}
        \label{Fig2}
    \end{figure}
    Indeed, as can be seen from Figure \ref{Fig2}, $f(n,R)$ changes sign, for instance, when $n=3,\,5,\,10$. In particular we underline that this phenomenon seems to occur for every $n\geq3$. In this case, we can conclude that neither concavity nor convexity properties apply to the functional $T_\gamma(\Omega)^\frac{1}{n+2}$, even if we restrict to sets that are centrally symmetric.

\Addresses
\end{document}